\documentclass[10pt,reqno]{amsart}
\usepackage[cp1251]{inputenc}
\usepackage[english]{babel}
\usepackage{amsmath}
\usepackage{amssymb}
\usepackage{amsfonts}
\usepackage{graphicx}
\usepackage{eucal}
\usepackage{hyperref,amsthm}
\RequirePackage{bbold}
\usepackage{mathptmx,bm}
\newtheorem{theorem}{Theorem}
\newtheorem{lemma}{Lemma}

\newtheorem{corollary}{Corollary}
\theoremstyle{definition}

\newtheorem{remark}{Remark}

\numberwithin{equation}{section}
\begin{document}

\title[On convergence rate to invariant measures in Markov Branching Processes]
    {On estimation of the convergence rate to \\ invariant measures in Markov Branching
    Processes \\ with possibly infinite variance and allowing immigration}

\author{{Azam~A.~IMOMOV}}
\address {Azam Abdurakhimovich Imomov
\newline\hphantom{iii} Karshi State University,
\newline\hphantom{iii} 17, Kuchabag st., 180100 Karshi city, Uzbekistan.}
\email{imomov{\_}\,azam@mail.ru}

\thanks{\copyright \ 2023 Imomov~A.A}

\subjclass[2020] {Primary 60J80; Secondary 60J85}

\keywords{Markov Branching Process; generating functions; immigration; transition functions;
        slowly varying function; invariant measures; convergence rate.}

\dedicatory{Dedicated to my Parents}

\begin{abstract}
    {The paper discusses the continuous-time Markov Branching Process allowing Immigration.
    We are considering a critical case for which the second moment of offspring law and the first
    moment of immigration law are possibly infinite. Assuming that the nonlinear parts of the appropriate
    generating functions are regularly varying in the sense of Karamata, we prove theorems on convergence
    of transition functions of the process to invariant measures. We deduce the speed rate of these
    convergence providing that slowly varying factors are with remainder.}
\end{abstract}

\maketitle

\section {Introduction}    \label{MySec:1}

    We continue our discussion of the population growth model called the continuous-time Markov Branching Process
    allowing Immigration (MBPI), which we considered in the paper {\cite{ImomovSFU14}}. Recall that this process can
    have a simple physical interpretation: the population size changes not only as a result of reproduction and
    disappearance of existing individuals, but also at the random stream of inbound ``extraneous'' individuals of
    the same type from outside. Namely, the process develops by the following scheme. Each individual existing at
    time $t \in {\mathcal{T}}:=[0, + \infty)$ independently of his history and of each other for a small time interval
    $(t, t+ \varepsilon)$ transforms into $j \in {\mathbb{N}}_0 \backslash \{ 1\} $ individuals with
    probability $a_j \varepsilon  + o(\varepsilon )$ and, with probability $1+ a_1 \varepsilon + o(\varepsilon)$
    stays to live or makes evenly one descendant (as $\varepsilon \downarrow 0$); where
    $\mathbb{N}_0=\{0\}\cup\mathbb{N}$ and $\mathbb{N}$ is the set of natural numbers. Here $\left\{{a_j} \right\}$
    are intensities of individuals' transformation that $a_j \ge 0$ for $j \in {\mathbb{N}}_0 \backslash \{1\}$ and
\[
    0 < a_0 < -a_1 = \sum\nolimits_{j \in {\mathbb{N}}_0 \backslash \{1\}}{a_j}<\infty. 
\]
    Independently of these
    for each time interval $j \in {\mathbb{N}}$ new individuals inter the population with probability
    $b_j \varepsilon + o(\varepsilon)$ and, immigration does not occur with probability $1 + b_0 \varepsilon + o(\varepsilon)$.
    Immigration intensities $b_j \ge 0$ for $j \in {\mathbb{N}}$ and 
\[
    0 <-b_0 = \sum\nolimits_{j \in {\mathbb{N}}}{b_j}<\infty.
\]
    Newly arrived individuals undergo transformation in accordance with the reproduction law generated by intensities
    $\left\{ {a_j} \right\}$; see {\cite[p.~217]{Sevast71}}. Thus, the process under consideration is completely
    determined by infinitesimal generating functions(GFs)
\begin{equation*}
    f(s) := \sum\limits_{j \in {\mathbb{N}}_0} {a_j s^j}
    \qquad  \mbox{and} \qquad  g(s) :=\sum\limits_{j \in {\mathbb{N}}_0}{b_j s^j}
    \qquad \parbox{2.4cm}{for {} $s\in{[0, 1)}$.}
\end{equation*}

    Denote $X(t)$ the population size at the time $t \in {\mathcal{T}}$ in MBPI. This is homogenous
    continuous-time Markov chain with state space $\mathcal{S}\subset\mathbb{N}_0$ and transition functions
\begin{equation*}
    p_{ij} (t):= \mathbb{P}_i \bigl\{ {X(t) = j} \bigr\}
    = \mathbb{P}\left\{X(t+\tau)=j \, \bigl| \, {X(\tau)=i} \bigr. \right\}
\end{equation*}
    for all $i,j \in {\mathcal{S}}$ and $\tau , t \in {\mathcal{T}}$.

    We devote the paper to the critical case only, i.e. $f'(1-)=\sum\nolimits_{j \in {\mathbb{N}}}{ja_j}=0$, and
    observe limit behaviours of transition functions $p_{ij} (t)$ as $t\to{\infty}$. Pakes~{\cite{PakesSankh}} was one
    of the first to study invariant measures for MBPI with finite variance and found an integral form of GF of invariant
    measures. He has proved that limits ${\pi}_j:=\lim_{t \to \infty}t^\lambda {p_{ij} (t)}$ exist independently on $j$, iff
    $\sum\nolimits_{j \in{\mathbb{N}}}{a_j j^2 \ln{j}}<\infty$ and $\sum\nolimits_{j \in {\mathbb{N}}}{b_j j\ln{j}}<\infty$,
    where $\lambda = {{2g'(1-)} \bigl/ {f''(1-)}\bigr.}$, besides the set $\left\{{\pi}_j, j\in{\mathcal{S}} \right\}$ presents
    an invariant measure for MBPI. In accordance with the results of the paper {\cite{LiPakes2012}} the invariant measure of
    MBPI can also be constructed by the strong ratio limit property of transition functions but slightly different in appearance.
    Namely, the set of positive numbers $\bigl\{\upsilon_j := \lim _{t \to \infty} {{p_{0j}(t)} / {p_{00} (t)}}\bigr\}$ is an
    invariant measure. Moreover it can be seen a close relation between the sets $\left\{{\pi}_j, j\in{\mathcal{S}} \right\}$
    and $\left\{\upsilon_j, j\in{\mathcal{S}}\right\}$, and their GFs $\pi{(s)} = \sum\nolimits_{j \in{\mathcal{S}}}{\pi_j s^j}$
    and $\mathcal{U}(s) = \sum\nolimits_{j \in{\mathcal{S}}}{\upsilon_j s^j}$. In fact, they are really only different versions
    of the same limit law. So, it is easy to see $\mathcal{U}(s)=\pi{(s)}\bigl/\pi{(0)}\bigr.$, and this is consistent with
    uniqueness, up to a multiplicative constant, of the invariant measure of MBPI.

    In this issue an exceptional interest represents an estimation of the rate of convergence to invariant measures. In
    our previous report~{\cite{ImomovSFU14}} the rate of convergence of $t^\lambda {p_{ij} (t)}$ to ${\pi}_j$, for all
    $i, j \in \mathcal{S}$, was studied under the condition $\max \bigl\{ f'''(1-), g''(1-) \bigr\} <\infty$. It was
    found ibidem that the convergence rate is $\mathcal{O}\bigl({{{\ln t}/ t}} \bigr)$ as $t\to{\infty}$.

    Throughout the paper, we adhere to the following conditions concerning
    $f(s)$ and $g(s)$ which are our \textit{Basic assumptions}:
\begin{equation*}
    f(s)=(1-s)^{1+\nu}\mathcal{L}\left({{{1} \over {1-s}}}\right),      \eqno[\textsf {$f_\nu$}]
\end{equation*}
    and
\begin{equation*}
    g(s)=-(1-s)^{\delta}{\ell}\left({{{1} \over {1-s}}}\right)          \eqno[\textsf {$g_\delta$}]
\end{equation*}
    for all $s\in [0, 1)$, where $0 < \nu , \delta < 1$ and $\mathcal{L}(\cdot)$, $\ell(\cdot)$ are
    \textit{slowly varying at infinity} (${\textbf{SV}}_\infty$) in the sense of Karamata; see for instance
    {\cite{Bingham}} and {\cite{SenetaRV}}. Basic assumptions imply that the offspring distribution belongs to
    the domain of attraction of the $(1+\nu)$-stable law, and the immigration distribution belongs
    to the domain of attraction of the $\delta$-stable law. Besides, by the criticality of our process,
    assumption $[f_\nu]$ implies that $2b:=f''(1-)=\infty$. If $b<\infty$ then representation $[f_\nu]$
    holds with $\nu =1$ and $\mathcal{L}(t) \to b$ as $t \to \infty$. Similarly, GF $g(s)$ of the form $[g_\delta]$
    generates the immigration law, having the $\delta$-order moment. But if $g'(1-)<\infty$ then the assumption
    $[g_\delta]$ will be fulfilled with $\delta =1$ and $\mathcal{\ell}(t) \to{g'(1-)}$ as $t \to \infty$.

    By perforce we allow to forcedly put forward an additional requirement for $\mathcal{L}(x)$ and ${\ell}(x)$. So we can write
\begin{equation*}
    {{\mathcal{L}\left( {\lambda x} \right)} \over {\mathcal{L}(x)}}
    = 1 + {\mathcal{O}}\bigl(\alpha(x)\bigr)
    \quad \parbox{2.2cm}{{as} {} $x  \to \infty$}        \eqno[\textsf {$\mathcal{L}_{\nu}$}]
\end{equation*}
    for each $\lambda > 0$, where $\alpha(x)$ is known positive decreasing function so that $\alpha(x) \to 0$ as
    $x \to \infty $. In this case $\mathcal{L}(x)$ is called ${\textrm{SV}}_\infty$ with remainder
    ${\mathcal{O}}\bigl(\alpha(x)\bigr)$; see {\cite[p.~185, condition SR1]{Bingham}}.
    Wherever we exploit the condition $\left[\mathcal{L}_{\nu} \right]$ we will suppose that
\begin{equation*}
    \alpha(x) = {\mathcal{O}}\left( {{{\mathcal{L}\left(x\right)} \over {x^\nu }}} \right)
    \quad \parbox{2.2cm}{{as} {} $x  \to \infty$.}
\end{equation*}
    Similarly, we also allow a condition
\begin{equation*}
    {{{\ell}\left( {\lambda x} \right)} \over {{\ell}(x)}}
    = 1 + {\mathcal{O}}\bigl(\beta(x)\bigr)
    \quad \parbox{2.2cm}{{as} {} $x  \to \infty$}     \eqno[\textsf {${\ell}_{\delta}$}]
\end{equation*}
    for each $\lambda > 0$, where
\begin{equation*}
    \beta(x) = {\mathcal{O}}\left( {{{{\ell}\left(x\right)} \over {x^\delta}}} \right)
    \quad \parbox{2.2cm}{{as} {} $x  \to \infty$.}
\end{equation*}

    In the preprint~{\cite{ImomovMeyliev20}} it was shown that the asymptotes of the transition functions depends on the sign
    of the parameter $\gamma := \delta - \nu$. In addition, the limit functions $U(s):=\lim_{t\to\infty}{\mathcal{P}}(t;s)$
    for $\gamma>0$ and $\pi(s):=\lim_{t\to\infty}e^{T(t)}{\mathcal{P}}(t;s)$ for $\gamma<0$ and for some $T(t)$ were found.

    In this report we attempt to deduce the speed rate of the convergence in theorems proved in
    {\cite{ImomovMeyliev20}} providing that conditions $[\mathcal{L}_{\nu}]$ and $[{\ell}_{\delta}]$ hold.

    The rest of this paper is organized as follows. Section~\ref{MySec:2} contains main results.
    Section~\ref{MySec:3} provides auxiliary statements that will be essentially used in the proof of our theorems.
    Section~\ref{MySec:4} is devoted to the proof of main results.

\section{Main Results}   \label{MySec:2}

    Put into consideration GF 
\[
    {\mathcal{P}}_i (t;s):=\sum\nolimits_{j \in {\mathcal{S}}}{p_{ij}(t)s^j}.
\]
    It is not difficult to see that (see {\cite{PakesSankh}})
\begin{equation}                    \label{2.1}
    {\mathcal{P}}_i (t;s) = \Bigl(F(t;s)\Bigr)^{i} \exp \left\{ {\int\limits_0^t {g\bigl({F(u;s)} \bigr)du}} \right\},
\end{equation}
    where $F(t;s)$ is GF of Markov Branching Process initiated by single individual without immigration. Since
    $F(t;s)\to 1$ as $t\to \infty$ uniformly in $s\in[0,d]$, $d<1$ (see Lemma~\ref{MyLem:1} below), it suffice
    to consider ${\mathcal{P}}(t;s):= {\mathcal{P}}_{0}(t;s)$. Then under Basic assumptions, due to the
    Kolmogorov's backward equation ${\partial{F}}\bigl/{\partial{t}}\bigr.=f\left(F\right)$, from \eqref{2.1} follows
\begin{equation}                    \label{2.2}
    {\mathcal{P}}(t;s) = \exp \left\{{\int\limits_s^{F(t;s)} {{{g(u)} \over {f(u)}}\,du}} \right\}.
\end{equation}
    In view of Basic assumptions, the integrand
\begin{equation}                    \label{2.3}
    {{g(u)} \over {f(u)}} = -{\left(1-u\right)^{\gamma -1} \textsf{\emph{L}}{\left({1} \over {1-u}\right)}},
\end{equation}
    where $\gamma := \delta - \nu$ and
\begin{equation*}
    \textsf{\emph{L}}(t) :={{\ell(t)} \over {\mathcal{L}(t)}} \,\raise0.8pt\hbox{.}
\end{equation*}
    All appearances, the three cases can be divided concerning the classification of $\mathcal{S}$, depending on a sign
    of $\gamma$. By virtue of \eqref{2.3}, integral $\int\nolimits_s^{1} {\left[{g(u)} \left/ {f(u)} \right.\right]du}$
    converges if $\gamma >0$, and diverges if $\gamma <0$. Thus and so, as it was shown in {{\cite{LiPakes2012}}, that
    $\mathcal{S}$ is \textit{positive-recurrent} if $\gamma > 0$, and it is \textit{transient} if $\gamma < 0$. The special
    case $\gamma = 0$ implies that $g(s)=f'(s)$ and that $\textsf{\emph{L}}(t) \to 1+\nu$ as $t \to \infty$. And we get
    another population process called \textit{Markov Q-process} instead of MBPI. We refer the reader to {\cite{Imomov17}}
    and {\cite{Imomov12}} for the details on the Markov Q-process; see also {\cite[pp.~56--58]{ANey}} and {\cite{Pakes99}}
    for the discrete-time case.

    Our main results appear only for the case $\gamma \neq 0$ in the following two theorems. Put
\begin{equation*}
    {\tau}(t) := {({\nu}t)^{1/\nu} \over {\mathcal{N}(t)}}
    \qquad \parbox{.8cm}{and} \qquad  T(t):={\bigl(\tau(t)\bigr)}^{|\gamma|},
\end{equation*}
    where ${\mathcal{N}}(x)$ is $\textrm{SV}_\infty$ defined in Lemma~\ref{MyLem:1} below.

\begin{theorem}                     \label{MyTh:1}
    Let $\gamma > 0$. Then ${\mathcal{P}}(t;s)$ converges to the function 
\[
    U(s)=\exp \left\{\int_s^{1}{\frac{g(u)}{f(u)}du}\right\}
\] 
    for $s \in [0, 1)$, and its power series expansion
    $U{(s)} = \sum\nolimits_{j \in{\mathcal{S}}}{u_j s^j}$ generates an invariant distribution
    $\left\{u_j, j\in{\mathcal{S}}\right\}$ for MBPI. The convergence is uniform over compact subsets
    of $[0, 1)$. In addition, if assumptions $[\mathcal{L}_{\nu}]$ and $[{\ell}_{\delta}]$ hold, then
\begin{equation}               \label{2.4}
    {\mathcal{P}}(t;s)= U(s)\left(1+{\Delta(t;s)}{{\mathcal{K}}} \bigl({\tau}(t)\bigr)\right),
\end{equation}
    where ${\mathcal{K}}(x)={\mathcal{L}}^{-{{\delta}/{\nu}}}(x) {\ell (x)}$ and the function
    ${\mathcal{N}}(x)$ is $\textrm{SV}_\infty$ defined in \eqref{3.2} below, herewith
\begin{equation*}
    \Delta(t;s) = {{\,1\,}\over {\gamma}}\,{{{1}} \over {\bigl(\lambda(t;s)\bigr)^{{\gamma} / \nu}}}
    +\mathcal{O} \left({{\ln \bigl[{\Lambda}({1-s})\lambda(t;s)\bigr]} \over
    {\bigl(\lambda(t;s)\bigr)^{{\delta / \nu}}}} \right)
    \quad \parbox{2.2cm}{{as} {} $t  \to \infty$,}
\end{equation*}
    where $\lambda(t;s)=\nu{t}+{\Lambda}^{-1}\left({1-s}\right)$
    and $\Lambda(y)=y^{\nu}\mathcal{L}\left(1/y\right)$.
    The transition functions are given by
\begin{equation}               \label{2.5}
    p_{ij} (t)= u_j \left(1+\mathcal{O} \left({{K(t)}\over{t^{\gamma/\nu}}}\right)\right)
    \quad \parbox{2.2cm}{{as} {} $t \to \infty$,}
\end{equation}
    where $K(t)$ is $\textrm{SV}_\infty$.
\end{theorem}

    Another asymptotic property comes out for ${\mathcal{P}}(t;s)$ when $\gamma < 0$.
    We can easily verify that under Basic assumptions
\begin{equation*}
    -{{\,\ln{p_{00}(t)}\,} \over {T(t)}} \sim {{\,1\,}\over{|\gamma|}}
    {\textsf{\emph{L}}{\bigl(\tau{(t)}\bigr)}}
    \quad \parbox{2.2cm}{{as} {} $t  \to \infty$.}
\end{equation*}

    This asymptotic formula shows that ${{\bigl({T(t)}\bigr)^{-1}} \ln{p_{00}(t)}}$ is asymptotically $\textrm{SV}_\infty$,
    and it suggests that we should look for a limit as $t \to \infty$ of the function $e^{T(t)}{\mathcal{P}}(t;s)$.
    First wee need to discuss ${\textrm{SV}}_\infty$ property of $\textsf{\emph{L}}(t)$. In accordance with
    Slowly varying theory, functions ${\ell}(\cdot)$ and $\mathcal{L}(\cdot)$ are positive. Then by virtue
    of {\cite[p.~185, Theorem~3.12.2~{(SR1)}]{Bingham}}, we can get the following propositions:
\begin{itemize}
\vspace{2.2mm}
\item[$\blacktriangleright$]  $[\mathcal{L}_{\nu}]$ \quad $\Longleftrightarrow$ \quad
                                ${\mathcal{L}(x)} = C_{\mathcal{L}} + {\mathcal{O}}\bigl(\alpha(x)\bigr)
                              \quad \parbox{2.2cm}{{as} {} $t \to \infty$,}   \hfill [C_{\mathcal{L}}]$
\vspace{2.2mm}
\item[$\blacktriangleright$]  $[{\ell}_{\delta}]$ \quad \, $\Longleftrightarrow$ \quad ${{\ell}(x)}
                              = C_{\ell} + {\mathcal{O}}\bigl(\beta(x)\bigr)
                              \quad \parbox{2.2cm}{{as} {} $t \to \infty$,}   \hfill [C_{\ell}] $
\vspace{2.2mm}
\end{itemize}
    where $C_{\mathcal{L}}, C_{\ell}$ are positive constants and functions $\alpha(x), \beta(x)$ are in
    $[\mathcal{L}_{\nu}]$ and $[{\ell}_{\delta}]$. We then can reveal the fact that
\begin{equation}                    \label{2.6}
    \textsf{\emph{L}}(t) = {{\ell(t)} \over {\mathcal{L}(t)}} =
    {C_{\textsf{\emph{L}}}} + \mathcal{O}\left({{{\ell}(t)} \over {t^{\delta}}}\right)
    \quad \parbox{2.2cm}{{as} {} $t \to \infty$,}
\end{equation}
    since $\delta < \nu$, where $C_{\textsf{\emph{L}}}={{C_{\ell}}\left/{C_{\mathcal{L}}}\right.}$. Note this
    requirement put forward for $\textsf{\emph{L}}(t)$ is quite possible. Especially, we reach an ``excellent result''
    in this issue, if we exclusively require that $C_{\textsf{\emph{L}}}=|\gamma|$. Namely the following explicit-form
    of $\pi(s)=\lim_{t\to\infty}e^{T(t)}{\mathcal{P}}(t;s)$ was found in {\cite{ImomovMeyliev20}}:
\begin{equation}                    \label{2.7}
    \pi(s) = \exp \left\{{{\,1\,}\over{(1-s)^{|\gamma|}}} + {\int\limits_s^1 {\left[ {{{g(u)} \over {f(u)}}
    + {|\gamma| \over {(1-u)^{1+|\gamma|}}}} \right]du}} \right\}.
\end{equation}
    Now in the following theorem we determine the convergence rate of $e^{T(t)}{\mathcal{P}}(t;s)$ to $\pi(s)$.

\begin{theorem}                     \label{MyTh:2}
    Let $\gamma < 0$ and $C_{\textsf{L}}=|\gamma|$ in \eqref{2.6}. If $\mu:={2\delta -\nu}>0$, then
\begin{equation}                    \label{2.8}
    {e^{T(t)}} {\mathcal{P}} (t;s) = {\pi}(s) \bigl(1+\rho(t;s)\bigr),
\end{equation}
    where $\rho(t;s)\to 0$ as $t \to \infty$ uniformly in $s\in [0, r]$,  $r<1$, and the
    limiting GF ${\pi}(s)$ can be expressed in the form of \eqref{2.7}.  In addition,
    if assumptions $[\mathcal{L}_{\nu}]$ and $[{\ell}_{\delta}]$ hold, then
\begin{equation}                    \label{2.9}
    \rho(t;s) = \mathcal{O}\left({{{\ell}\left(\tau(t)\right)} \over {{\left(\tau(t)\right)}^{\mu}}}\right)
    \quad \parbox{2.2cm}{{as} {} $t \to \infty$}
\end{equation}
    uniformly in $s\in [0, r]$, $r<1$. Denoting the power series expansion of ${\pi}(s)$
    by $\sum\nolimits_{j \in{\mathcal{S}}}{\pi_j s^j}$, transition functions are given by
\begin{equation}               \label{2.10}
    p_{ij} (t)= \pi_j \left(1+\mathcal{O}\left({{{\ell}\left(\tau(t)\right)}
    \over {{\left(\tau(t)\right)}^{\mu}}}\right)\right)
    \quad \parbox{2.2cm}{{as} {} $t \to \infty$,}
\end{equation}
    and $\left\{\pi_j, j\in{\mathcal{S}}\right\}$ is an invariant measure for MBPI.
\end{theorem}

\begin{remark}
    An appeared form of limiting GF ${\pi}(s)$ found in first part of Theorem~\ref{MyTh:2} is compatible
    with the results of the papers {\cite{PakesSankh}} and {\cite{ImomovSFU14}} established for the
    case of $\max \bigl\{ f''(1-), g'(1-) \bigr\} <\infty$. Thus this theorem essentially
    strengthens last-mentioned results.
\end{remark}

\begin{remark}
    The conditions $C_{\textsf{L}}=|\gamma|$ and $\mu > 0$ in Theorem~\ref{MyTh:2} are essential, because
    they ensure the convergence of the integral in \eqref{2.7}. In fact, due to Basic assumptions and
    \eqref{2.6}, an expression  $(1-u)^{\mu -1}$ be a majorizing function for the integrand. So that the function
\begin{equation}                    \label{2.11}
    {\mathcal{B}}(s): = \exp \left\{ {\int_s^1 {\left[ {{{g(u)} \over {f(u)}}
    + {|\gamma| \over {(1-u)^{1+|\gamma|}}}} \right]du}} \right\}
\end{equation}
    is bounded for $s \in [0, 1]$.
\end{remark}

    The following result is consequent of Theorem~\ref{MyTh:2}.
\begin{corollary}                        \label{MyCor:1}
    Under the conditions of Theorem~\ref{MyTh:2}
\begin{equation*}
    {e^{T(t)}}p_{00} (t) = {\mathcal{B}(0)} \left(1+\mathcal{O}
    \left({{{\ell}\left(\tau(t)\right)} \over {{\left(\tau(t)\right)}^{\mu}}}\right)\right)
    \quad \parbox{2.2cm}{{as} {} $t  \to \infty$,}
\end{equation*}
    where the function ${\mathcal{B}}(s)$ is defined in \eqref{2.11}.
\end{corollary}

\begin{remark}
    Further reasonings imply, that in estimations of error terms of asymptotic relations in
    Theorems above, the functions $\mathcal{L}(x)$ and $\ell(x)$ can be missed because owing
    to assertions $[C_{\mathcal{L}}]$ and $[C_{\ell}]$ they are asymptotically constant.
\end{remark}

\section{Auxiliaries}   \label{MySec:3}

    In this section, we provide some auxiliary assertions that will be essentially used in the proof of our theorems.

    First we are interested in asymptotic representation of GF of Markov branching processes $Z(t)$ without
    immigration. Let $F(t;s)=\mathbb{E}\left[s^{Z(t)} \bigl| {Z(0)=1} \bigr. \right]$ is GF of the process
    initiated by single individual. Put $R(t;s):=1-F(t;s)$. The following result called the Basic lemma of
    the theory of critical Markov branching processes, has been proved in {\cite{Imomov17}}. We provide
    it in slightly different form below, taking from its proof.

\begin{lemma}                \label{MyLem:1}
    If the condition $[f_\nu]$ holds then
\begin{equation}                     \label{3.1}
    {{1} \over {R(t;s)}} = {{({\nu}t)^{1/\nu}} \over {\mathcal{N}(t)}}
    \cdot \left[{1 + {{\mathcal{M}(s)} \over {t}}} \right]^{1/\nu}
\end{equation}
    for all $s\in [0, 1)$, where ${\mathcal{N}}(x)$ is ${SV}_\infty$ such that
\begin{equation}                     \label{3.2}
    {\mathcal{N}}^{\,\nu}(t) \cdot \mathcal{L}\left({{{({\nu}t)^{{1/\nu}}}
    \over {\mathcal{N}(t)}}} \right) \longrightarrow 1
    \quad \parbox{2.2cm}{{as} {} $t \to \infty$,}
\end{equation}
    and $\mathcal{M}(s)$ is GF of invariant measures of MBP having the form of
\begin{equation*}
    \mathcal{M}(s) = \int\limits_1^{{{1} / {(1-s)}}}
    {{{dx} \over {x^{1-\nu }\mathcal{L}(x)}}}\,\raise0.9pt\hbox{.}
\end{equation*}
\end{lemma}

    Let's introduce a function
\begin{equation*}
    \Lambda (y):= y^\nu \mathcal{L}\left( {{{\,1\,} \over {y}}} \right) = {{f(1-y)} \over {y}}
\end{equation*}
    for $y\in (0, 1]$. Note that the function $y\Lambda (y)$ is positive, tends to zero and has a monotone derivative so that
    ${{y\Lambda '(y)}\mathord{\left/{\vphantom {{y\Lambda '(y)}{\Lambda (y)}}}\right.\kern-\nulldelimiterspace}{\Lambda (y)}}\to\nu$
    as $y \downarrow 0$; see {\cite[p.~401]{Bingham}}. Thence it is natural to write
\begin{equation}                     \label{3.3}
    {{y\Lambda '(y)} \over {\Lambda (y)}} = \nu + \delta(y),
\end{equation}
    where $\delta(y)$ is continuous and $\delta(y) \to 0$ as $y \downarrow 0$.
    Since $\Lambda (1)=\mathcal{L}(1)=a_0$ formula \eqref{3.3} yields
\begin{equation*}
    \Lambda(y)=a_{0}y^{\nu} \exp{\int\limits_1^y {{{\delta(u)}\over{u}}du}}.
\end{equation*}
    Therefore we have
\begin{equation*}
    \mathcal{L}\left( {{{\,1\,} \over {y}}} \right)
    = a_{0} \exp {\int\limits_1^y {{{\delta(u)} \over {u}}du}}.
\end{equation*}
    Substituting $u ={{1}/{t}}$ in last integrand gives
\begin{equation*}
    \mathcal{L}\left( x \right)
    = a_{0} \exp {\int\limits_1^x {{{\varepsilon (t)} \over {t}}dt}},
\end{equation*}
    where $\varepsilon(t)=-\delta(1/t)$ and $\varepsilon(t) \to 0$ as $t \to \infty$.
    Combination of last equation with $[\mathcal{L}_{\nu}]$ entails
\begin{equation*}
    \int\limits_{x}^{\lambda x} {{{\varepsilon (t)} \over t}dt}
    = \ln \left[1 + {\mathcal{O}}\bigl(\alpha(x)\bigr)\right] = {\mathcal{O}}\bigl(\alpha(x)\bigr)
    \quad \parbox{2.2cm}{\textit{as} {} $x  \to \infty$}
\end{equation*}
    for each $\lambda > 0$. Applying the mean value theorem to the left-hand side of last equality, we can assert
    that $\varepsilon(x)={\mathcal O}\left({\alpha(x)}\right)$ and thus the condition $[\mathcal{L}_{\nu}]$ entails
\begin{equation}                \label{3.4}
    \delta (y) = {\mathcal O}\left( {\alpha \left({\,1\,} \over {y}\right)} \right)
    \quad \parbox{2.0cm}{\textit{as} {} $y  \downarrow 0$.}
\end{equation}

    The following result is a modification of Lemma~\ref{MyLem:1}
    and it also will essentially be required in our discussions.

\begin{lemma}                 \label{MyLem:2}
    Let assumptions $[f_\nu]$ and $[\mathcal{L}_{\nu}]$ hold. Then
\begin{equation}          \label{3.5}
    {{1} \over {\Lambda \left(R(t;s)\right)}}-{{1}\over {\Lambda \left(1-s\right)}}
    =\nu{t} + {\mathcal{O}}{\bigl(\ln{\nu(t;s)}\bigr)}
    \quad \parbox{2.2cm}{{as} {} $t  \to \infty$,}
\end{equation}
    where $\nu(t;s)={\Lambda{\left(1-s\right)}}\nu{t}+1$.
\end{lemma}

\begin{proof}
    From \eqref{3.3} we write
\begin{equation}          \label{3.6}
    {{R\Lambda{'}\left(R\right)} \over {\Lambda \left(R\right)}} = \nu + \delta\left(R\right)
\end{equation}
    since $R:=R(t;s)\to{0}$ as $t\to{\infty}$. By the backward Kolmogorov equation ${\partial{F} \mathord{\left/{\vphantom
    {{\partial{F}}{\partial{t}}}}\right.\kern-\nulldelimiterspace}{\partial{t}}}=f\left({F}\right)$
    and considering representation $\left[{f_{\nu}}\right]$, the relation \eqref{3.6} becomes
\begin{equation*}
    {{d\Lambda\left(R\right)} \over {dt}}=-{{\Lambda\left(R\right)}\over {R}}
    f\left(1-R\right)\bigl(\nu + \delta\left(R\right)\bigr)=-{{\Lambda^{2}\left(R\right)}}
    \bigl(\nu + \delta\left(R\right)\bigr).
\end{equation*}
    Therefore
\begin{equation}          \label{3.7}
    {d}\left[{{1} \over {\Lambda\left(R\right)}} - \nu{t}\right]= \delta\left(R\right)dt.
\end{equation}
    Integrating \eqref{3.7} over $[0, t)$ we reach to the following equation:
\begin{equation}               \label{3.8}
    {{1}\over{\Lambda \left(R(t;s)\right)}}-{{1}\over {\Lambda \left(1-s\right)}}
    = \nu{t} + {\int\limits_0^t {{\delta\left(R(u;s)\right)}du}},
\end{equation}
    where $\delta(y)$ is in \eqref{3.3}. Now we should calculate integral in \eqref{3.8}.
    Considering \eqref{3.4} we write
\begin{equation}          \label{3.9}
    {\int\limits_0^t {{\delta\left(R(u;s)\right)}du}}
    = {\int\limits_0^t {{\mathcal{O}}{\bigl({\Lambda\left(R(u;s)\right)}\bigr)}du}}.
\end{equation}
    Needless to say that $R(t;s)\to {0}$ as $t \to \infty$ uniformly in $s\in [0, 1)$, due
    to \eqref{3.1}. Therefore, since $\Lambda(y) \to 0$ as $y \downarrow 0$, the integral
    in the right-hand side of \eqref{3.9} is ${o}(t)$ as $t \to \infty$. Hence
\begin{equation*}
    {\Lambda \left(R(t;s)\right)}={{{1}} \over {\lambda(t;s)}}
    + {o}{\left( {{{1}} \over {\lambda(t;s)}} \right)}
    \quad \parbox{2.2cm}{\textit{as} {} $t  \to \infty$,}
\end{equation*}
    where $\lambda(t;s)=\nu{t}+{\Lambda}^{-1}\left({1-s}\right)$. Therefore
\begin{equation*}
    {\int\limits_0^t {{\mathcal{O}}{\bigl({\Lambda\left(R(u;s)\right)}\bigr)}du}}
    ={\mathcal{O}}\left({\int\limits_0^t {{\Lambda\left(R(u;s)\right)}du}}\right)
    = {\mathcal{O}}\bigl({\ln{\nu(t;s)}}\bigr)
    \quad \parbox{2.2cm}{\textit{as} {} $t  \to \infty$.}
\end{equation*}
    This together with \eqref{3.8} and \eqref{3.9} entails \eqref{3.5}.
\end{proof}

\begin{lemma}              \label{MyLem:3}
    Let ${L}(t)$ is $\textrm{SV}_\infty$ with remainder $\varrho(t)$. Then for $\sigma > 0$
\begin{equation}                    \label{3.10}
    {\int\limits_{t}^{\infty} {{y^{-(1+\sigma)}{L}(y)dy}}} = {{\,1\,} \over {\sigma}}
    {1 \over {\,t^{\sigma}}} {L}(t) \left({1+ \mathcal{O}\bigl(\varrho(t)\bigr)}\right)
    \quad \parbox{2.2cm}{{as} {} $t \to \infty$.}
\end{equation}
\end{lemma}

\begin{proof}
    Undoubtedly $\int_{1}^{\infty}{u^{-(1+\sigma)}du}=1/{\sigma}$. Considering this fact
    and making the substitution $y:=ut$ in the integrand of \eqref{3.10}, we write
\begin{equation}                    \label{3.11}
    {\int\limits_{t}^{\infty} {{y^{-(1+\sigma)}{L}(y)dy}}}
    = {{\,1\,}\over{\sigma}} {{L(t)} \over {t^{\sigma}}} \left[1 + {\sigma}
    \int\limits_{1}^{\infty}{\left[{{L(ut)} \over {L(t)}}-1 \right] u^{-(1+\sigma)}du}\right].
\end{equation}
    By definition of $\textrm{SV}_\infty$-function with remainder, the expression in brackets of
    integrand on right-hand side of \eqref{3.11} tends to $0$ as $t \to \infty$ uniformly in $u>{1}$
    (by Uniform Convergence Theorem for $\textrm{SV}_\infty$-functions~{\cite[Theorem~1.5.2]{Bingham}})
    with the speed rate $\mathcal{O}\bigl(\varrho(t)\bigr)$. Thus we have \eqref{3.10}.

    The Lemma is proved.
\end{proof}

\begin{lemma}              \label{MyLem:4}
    Let conditions $[\mathcal{L}_{\nu}]$ and $[{\ell}_{\delta}]$ hold and $\gamma > 0$. Then
\begin{equation}              \label{3.12}
    \int\limits_x^1 {{{g(u)}\over {f(u)}}\,du} = {{\,1\,}\over {\gamma}}{{g(x)} \over {\,\Lambda{\left({1-x} \right)}}}
    \left(1+\mathcal{O}\bigl({\Lambda{\left({1-x} \right)}}\bigr)\right)
    \quad \parbox{2cm} {\textit{as} {} $x \uparrow 1$.}
\end{equation}
\end{lemma}

\begin{proof}
    It follows from Basic assumption that
\begin{equation}              \label{3.13}
    \mathcal{I}(x) := \int\limits_x^1 {{{g(u)} \over {f(u)}}\,du}
    = -{\int\limits_{{1}/{(1-x)}}^{\infty}} {y^{-(1+\gamma)}{\textsf{\emph{L}}}(y)dy},
\end{equation}
    where $\textsf{\emph{L}}(t) = {{\ell(t)} / {\mathcal{L}(t)}}$ as before.
    In our conditions, we can easily make sure that
\begin{equation*}
    {{{\textsf{\emph{L}}}(ut)} \over {{\textsf{\emph{L}}}(t)}}-1
    = \mathcal{O}\left({{{\mathcal{L}}(t)} \over {t^{\nu}}}\right)
    \quad \parbox{2.2cm} {\textit{as} {} $t \to \infty$}
\end{equation*}
    uniformly in $u>0$. On right-hand side of \eqref{3.13} we can directly use of the formula~\eqref{3.10}
    with $t= {{1}/{(1-x)}}$ and $r(t)=\mathcal{O}\left({{{\mathcal{L}}(t)} \bigl/ {t^{\nu}} \bigr.}\right)$.
    Then
\begin{equation*}
    \mathcal{I}(x) = -{{\,1\,} \over {\gamma}}{{\textsf{\emph{L}}(t)} \over {t^{\gamma}}}
    \left[1 + \mathcal{O}\left({{\mathcal{L}(t)} \over {t^{\nu}}}\right) \right]
    \quad \parbox{2.2cm} {\textit{as} {} $t \to \infty$.}
\end{equation*}
    Now returning to primary designations, we get to \eqref{3.12}.

    The Lemma is proved.
\end{proof}

\section{Proof of Theorems}   \label{MySec:4}

    In this final section we consistently prove the Main results.

\begin{proof} [Proof of Theorem~\ref{MyTh:1}]
    Rewrite \eqref{2.2} as follows:
\begin{equation}                    \label{4.1}
    {\mathcal{P}}(t;s) = U(s) \exp \left\{{\int\limits_1^{F(t;s)} {{{g(u)} \over {f(u)}}\,du}} \right\},
\end{equation}
    where
\begin{equation}                    \label{4.2}
    U(s) = \exp \left\{{\int\limits_s^{1} {{{g(u)} \over {f(u)}}\,du}} \right\}.
\end{equation}
    First we see that due to \eqref{2.3} the integral in \eqref{4.1} converges for $s \in [0, 1)$ and becomes $0$
    as $t \to \infty$, therefore ${\mathcal{P}}(t;s)$ converges to $U(s)$ as $t \to \infty$ uniformly over compact
    subsets. Now, using the functional equation $F(t+\tau;s) = {F\bigl(t; {F(\tau;s)}\bigr)}$
    (see {\cite[p.~134]{PakesSankh}}), it follows
\begin{eqnarray*}
    {\mathcal{P}}(t+\tau;s)
    & = & {\mathcal{P}}(\tau;s)\cdot\exp \left\{ {\int\limits_{\tau}^{t+\tau}
    {g\left( {F(u;s)} \right)du}} \right\} \\
\nonumber\\
    & = & {\mathcal{P}}(\tau;s) \cdot \exp \left\{ {\int\limits_{0}^{t}
    {g\left( {F\bigl(u; {F(\tau;s)}\bigr)} \right)du}} \right\}
    =  {\mathcal{P}}(\tau;s)\cdot {\mathcal{P}}\bigl(t; {F(\tau;s)}\bigr),
\end{eqnarray*}
    and taking limit as $t \to \infty$ we have the following Schr\"{o}der type functional equation:
\begin{equation}                    \label{4.3}
    {U}\bigl(F(\tau;s)\bigr) = {\frac{1}{{\mathcal{P}}(\tau;s)}} \, {U}(s)
    \qquad \parbox{2.8cm}{{for any} {} $\tau \in {\mathcal{T}}$.}
\end{equation}
    Writing the power series expansion ${U}(s)= \sum\nolimits_{j \in{\mathcal{S}}}{{u}_j s^j}$, the equation \eqref{4.3}
    entails an invariant property ${u}_j = \sum\nolimits_{i \in {\mathcal{S}}} {{u}_i p_{ij}(\tau)}$. Obviously ${U}(1-)=1$
    and hence the function in \eqref{4.2} generates an invariant distribution $\left\{u_j, j\in{\mathcal{S}}\right\}$ for MBPI.

    Now we pass to the proof of \eqref{2.4}. Using \eqref{3.12} in \eqref{4.1} we obtain
\begin{equation}                    \label{4.4}
    {\mathcal{P}}(t;s) = U(s) \exp \Bigl\{{-I(t;s)} \Bigr\},
\end{equation}
    where
\begin{equation}                    \label{4.5}
    I(t;s) = {{\,1\,}\over {\gamma}}{{g{\bigl(F(t;s)\bigr)}} \over {\Lambda{\bigl(R(t;s)\bigr)}}}
    \left(1+\mathcal{O}\Bigl({\Lambda{\left(R(t;s)\right)}}\Bigr)\right)
    \quad \parbox{2.2cm}{{as} {} $t  \to \infty$.}
\end{equation}
    Next, we have to use the asymptotic expansion of $R(t;s)$. Relation \eqref{3.5} implies
\begin{equation}                    \label{4.6}
    {{1} \over {\Lambda \left(R(t;s)\right)}} = \lambda(t;s)
    \left(1+\mathcal{O}\left({{\ln{\left[\Lambda{(1-s)}\lambda(t;s)\right]}} \over {\lambda(t;s)}}\right)\right)
    \quad \parbox{2.2cm}{{as} {} $t  \to \infty$,}
\end{equation}
    and therefore
\begin{equation}                    \label{4.7}
    {R(t;s)} = {{\mathcal{N}(t;s)} \over {\left(\lambda(t;s)\right)^{1/\nu}}}
    \left(1+\mathcal{O}\left({{\ln{\left[\Lambda{(1-s)}\lambda(t;s)\right]}} \over {\lambda(t;s)}}\right)\right)
    \quad \parbox{2.2cm}{{as} {} $t  \to \infty$,}
\end{equation}
    where $\lambda(t;s)=\nu{t}+{\Lambda}^{-1}\left({1-s}\right)$ and $\mathcal{N}(t;s)={\mathcal{L}^{-{1/\nu}}}
    \bigl(1/{R(t;s)}\bigr)$. Since $g(s)$ is form of $[g_\delta]$, using \eqref{4.6} and \eqref{4.7} we obtain
\begin{equation}                    \label{4.8}
    {{g{\bigl(F(t;s)\bigr)}} \over {\Lambda{\bigl(R(t;s)\bigr)}}} = -{{\mathcal{N}^{\delta}(t;s)}
    \over {\left(\lambda(t;s)\right)^{{\gamma}/{\nu}}}} {\ell{\left({1}\over{R(t;s)}\right)}}
    \left(1+\mathcal{O}\left({{\ln{\left[\Lambda{(1-s)}\lambda(t;s)\right]}}\over{\lambda(t;s)}}\right)\right)
\end{equation}
    as $t \to \infty$. It it is easy to verify that the function $\mathcal{N}(t;s)$ is asymptotically
    equivalent to the $\textrm{SV}_\infty$-function $\mathcal{N}(t)$ defined in Lemma~\ref{MyLem:1}.

    Asymptotic formula \eqref{2.4} now follows from a combination of formulas \eqref{4.4}, \eqref{4.5}
    and \eqref{4.8}. Equation \eqref{2.5} follows from the continuity theorem for power series.

    The Theorem is proved.
\end{proof}

\begin{proof} [Proof of Theorem~\ref{MyTh:2}]
    We write
\begin{eqnarray}                    \label{4.9}
    {e^{T(t)}} {\mathcal{P}} (t;s)
    & = & \exp \left\{ {\bigl(\tau(t)\bigr)^{|\gamma|}} + {\int\limits_0^{t}
    {g\left( {F(u;s)} \right)du}} \right\}   \nonumber \\
    \nonumber  \\
    & = & \exp \left\{ \Delta(t;s) + {\bigl(\tau(t;s)\bigr)^{|\gamma|}} +
    {\int\limits_s^{F(t;s)} {{{g(x)} \over {f(x)}}\,dx}} \right\},
\end{eqnarray}
    where $\Delta(t;s)={\bigl(\tau(t)\bigr)^{|\gamma|}} -{\bigl(\tau(t;s)\bigr)^{|\gamma|}}$
    and ${\tau(t;s)}={R^{-1}(t;s)}$. A standard integration method yields
\begin{equation*}
    {\bigl(\tau(t;s)\bigr)^{|\gamma|}} = {{\,1\,}\over{(1-s)^{|\gamma|}}}
    + {\int\limits_s^{F(t;s)} {{{|\gamma| \over {(1-u)^{1+|\gamma|}}}} du}}.
\end{equation*}
    Therefore the relation \eqref{4.9} can be written as follows:
\begin{equation}                    \label{4.10}
    {e^{T(t)}} {\mathcal{P}} (t;s) = \pi(s) \cdot \exp \left\{\Delta(t;s)
    - {\int\limits_{F(t;s)}^{1} {\left[ {{{g(u)} \over {f(u)}}
    + {|\gamma| \over {(1-u)^{1+|\gamma|}}}}\right]du}} \right\},
\end{equation}
    where $\pi(s)$ has the form of \eqref{2.7}. An exponential factor in \eqref{4.10} defines the
    convergence rate $\rho(t;s)$ in \eqref{2.8}. Let's first evaluate $\Delta(t;s)$ as $t \to \infty$.
    According to Lemma~\ref{MyLem:1} $\mathcal{M}(0)=0$ and therefore ${\tau(t)}={\tau}(t;0)$.
    Hence, an asymptotic representation~\eqref{3.1} produces
\begin{eqnarray*}
    \Delta(t;s)
    & = & {\bigl(\tau(t)\bigr)^{|\gamma|}}\left[1-\left({1+{{\mathcal{M}(s)}
    \over {t}}}\right)^{|\gamma|/\nu}\right]          \nonumber \\
\nonumber \\
    & \sim & - |\gamma|{\bigl(\tau(t)\bigr)^{|\gamma|}}{{\mathcal{M}(s)} \over {{\nu}t}}
    = - |\gamma|{{\mathcal{M}(s)} \over {({\nu}t)^{\delta/\nu} {\mathcal{N}^{|\gamma|}(t)}}}
    \quad \parbox{2.2cm}{{as} {} $t \to \infty$.} \nonumber
\end{eqnarray*}
    On the other hand $\mathcal{M}(s)$ is bounded for $s\in [0, r]$,  $r<1$. Thus
\begin{equation}                    \label{4.11}
    \Delta(t;s)=\mathcal{O}\left({{\mathcal{L}_{\gamma}(t)} \over {t^{\delta/\nu}}}\right)
    \quad \parbox{2.2cm}{{as} {} $t \to \infty$,}
\end{equation}
    uniformly in $s\in [0, r]$,  $r<1$, where ${\mathcal{L}_{\gamma}(t)}={\mathcal{N}^{-|\gamma|}(t)}$.

    Now observe the integral in \eqref{4.10}. By virtue of relations \eqref{2.3} and \eqref{2.6}, the
    integrand in brackets becomes $\mathcal{O}\left( (1-u)^{\mu -1}{\ell}\bigl({1}/{(1-u)}\bigr)\right)$
    in the neighbourhood of the point $u=1$. So we have to examine the integral 
\[
    \int_{F(t;s)}^1 {(1-u)^{\mu -1}{\ell}\bigl({1}/{(1-u)}\bigr)du}
\] 
    as $t \to \infty$. Make the substitution $y=(1-u)^{-1}$
    to obtain the alternative form
\begin{equation*}
    {\int\limits_{{1}/{R(t;s)}}^{\infty}} {y^{-(1+\mu)}{\ell}(y)dy}.
\end{equation*}
    The direct application of Lemma~\ref{MyLem:3} transforms the last integral to the form
\begin{equation*}
    {\int\limits_{{1}/{R(t;s)}}^{\infty}} {y^{-(1+\mu)}{\ell}(y)dy}= {{\,1\,} \over {\mu}}
    R^{\mu}(t;s) {\ell}\left({1 \over {R(t;s)}}\right) \bigl({1+ o(1)}\bigr)
    \quad \parbox{2.2cm}{{as} {} $t \to \infty$.}
\end{equation*}
    But $R(t;s)={{\tau}^{-1}(t;s)}$ and $\tau(t;s){\tau}^{-1}(t)\to 1$
    as $t \to \infty$ uniformly in $s\in [0, 1)$. Thus
\begin{equation}                    \label{4.12}
    {\int\limits_{F(t;s)}^{1} {\left[ {{{g(u)} \over {f(u)}} +
    {|\gamma| \over {(1-u)^{1+|\gamma|}}}}\right]du}} =
    \mathcal{O}\left({{{\ell}\left(\tau(t)\right)} \over {{\left(\tau(t)\right)}^{\mu}}}\right)
    \quad \parbox{2.2cm}{{as} {} $t \to \infty$.}
\end{equation}

    Since $\mu < \delta$, comparing relations \eqref{4.11} and \eqref{4.12} gives that $\Delta(t;s)$
    decreases to zero faster than last integral, i.e. 
\[
    \Delta(t;s)={o}\left({{{\ell}\left(\tau(t)\right)} \bigl/{{\left(\tau(t)\right)}^{\mu}}} \bigr. \right) 
\] 
    as $t \to \infty$. So, considering together
    \eqref{4.10}--\eqref{4.12} we reach to asymptotic relation \eqref{2.8} with the error part $\rho(t;s)$
    in the form~\eqref{2.9}. Equation \eqref{2.10} follows from the continuity theorem for power series.

    Finally, we can verify that the function $\pi(s)$ satisfies the equation~\eqref{3.3}. Therefore  denoting its
    power series representation by ${\pi}(s)= \sum\nolimits_{j \in{\mathcal{S}}}{\pi_j s^j}$, we have an invariant
    property ${\pi}_j=\sum\nolimits_{i \in{\mathcal{S}}}{{\pi}_i p_{ij}(\tau)}$ for any $\tau >0$. Thus
    $\left\{{\pi}_j, j \in{\mathcal{S}}\right\}$ is an invariant measure for MBPI $X(t)$.

    The Theorem is proved.
\end{proof}

\begin{proof} [Proof of Corollary~\ref{MyCor:1}]
    The statement is immediately obtained from \eqref{2.8} setting $x = 0$ there.
\end{proof}

{\bf Acknowledgment.} The author is deeply grateful to the anonymous referee for his careful reading of the manuscript and
        for his kindly comments which contributed to improving the paper.

\end{document}